\documentclass[10pt]{article}  

\usepackage{a4wide}

\usepackage{graphicx}
\usepackage{amssymb}
\usepackage{amsmath}
\usepackage{amsthm}
\usepackage{amsfonts}
\usepackage{xcolor}

\newtheorem{theorem}{Theorem}[section]
\newtheorem{proposition}[theorem]{Proposition}

\newtheorem{lemma}[theorem]{Lemma}
\newtheorem{definition}[theorem]{Definition}
\newtheorem{remark}[theorem]{Remark}

\newcommand{\Pmk}{\mathcal{P}^-_{k}}
\newcommand{\Ppk}{\mathcal{P}^+_{k}}
\newcommand{\Ppmk}{\mathcal{P}^{\pm}_{k}}
\newcommand{\Ppo}{\mathcal{P}^+_{1}}
\newcommand{\R}{\mathbb R}
\newcommand{\RN}{\mathbb R^N}
\newcommand{\SN}{\mathcal{S}^N}
\newcommand{\CR}{\mathcal{C}_R}
\newcommand{\tr}{{\rm Tr}}

\title{On positive solutions of fully nonlinear degenerate Lane-Emden type equations}
\author{Giulio Galise\footnote{Dip. di Matematica \lq\lq G. Castelnuovo\rq\rq, Sapienza Universit\`a di Roma,  P.le Aldo Moro 2, I–00185 Roma, Italy \texttt{galise@mat.uniroma1.it}
}}
\date{}
\begin{document}

\maketitle

\begin{abstract}
\noindent
We prove existence and uniqueness results of positive viscosity solutions of fully nonlinear degenerate elliptic equations with power-like zero order perturbations in bounded domains.  The principal part of such equations is either $\Pmk(D^2u)$ or  $\Ppk(D^2u)$,  some sort of \lq\lq truncated Laplacians\rq\rq,  given respectively by the smallest and the largest partial sum of  $k$ eigenvalues of the Hessian matrix. New phenomena with respect to the semilinear case  occur. Moreover, for $\Pmk$, we explicitely find the critical exponent $p$ of the  power nonlinearity that separates the  existence and nonexistence range of nontrivial solutions with zero Dirichlet boundary condition.
\end{abstract}

\vspace{1cm}



\section{Introduction and main results}
The aim of this paper is to investigate the existence and uniqueness  of nontrivial solutions of  some nonproper fully nonlinear very degenerate elliptic equations of second order. More specifically we are interested in nonnegative solutions of 
\begin{equation}\label{eqbase-}
\Pmk(D^2u(x))+u^p(x)=0\quad\text{in $\Omega$}
\end{equation}
or of
\begin{equation}\label{eqbase+}
\Ppk(D^2u(x))+u^p(x)=0\quad\text{in $\Omega$},
\end{equation}
subject to the  Dirichlet boundary condition $u=0$ on $\partial\Omega$. From now on $\Omega$ will be a  bounded domain  of $\RN$, $N\geq2$ and $p>0$. In \eqref{eqbase-}-\eqref{eqbase+}  the operators $\Ppmk$ are defined for $C^2$-functions by the partial sums
\begin{equation}
\Pmk(D^2u)=\sum_{i=1}^k\lambda_i(D^2u)\qquad\text{and}\qquad\Ppk(D^2u)=\sum_{i=N-k+1}^N\lambda_i(D^2u)
\end{equation}
of the ordered eigenvalues $\lambda_1(D^2u)\leq\ldots\leq\lambda_N(D^2u)$ of the Hessian matrix $D^2u$. They  reduce to the standard Laplacian $\Delta$ if $k=N$. 
 If $u$ is merely a continuous function the previous definitions have to be understood in the weak viscosity sense, see e.g. \cite{CC,CIL,IL}.

\smallskip
The interest for such operators originated in geometric framework since the works of Sha \cite{Sha1,Sha2} and Wu \cite{Wu},  dealing with compact Riemannian manifolds which are $k$-convex, i.e.  
if the partial sum of the principal curvature functions
$$
\lambda_1+\ldots+\lambda_k
$$
is positive. More recent results in $k$-convex geometry can be found in Harvey and Lawson \cite{HL2}.\\
The operators $\Ppmk$ also arise  in the papers of  Ambrosio and Soner \cite{AS}, where the authors developed a level  set  approach for  the  mean  curvature  evolution  of  surfaces with  arbitrary   co-dimension, as well as in Oberman and Silvestre \cite{OS} concerning the underlying Dirichlet problem satisfied by the convex envelope of a prescribed boundary function. In the PDE context, we wish to mention the works of Caffarelli, Li and Nirenberg \cite{CLN,CLN2} which deal with maximum principle and symmetry type results for singular solutions, and of Harvey and Lawson \cite{HL} about existence and uniqueness of solutions of more general pure second order equations of the form $F(D^2u)=0$. Further developments can be found in \cite{CDLV,CPa,Vitolo,AGV,GV}.	
In the recent paper \cite{BGI}, the case $p=1$ is considered, leading to the eigenvalue problem for $\Ppmk$. There the authors extend  the notion of generalized principal eigenvalue to such operators, in the spirit of the acclaimed work of Berestycki-Nirenberg-Varadhan \cite{BNV}, shedding light on some very unusual phenomena due to the high degeneracy of $\Ppmk$. For instance, the failure of the strong minimum principle for $\Pmk$ (resp. the strong maximum principle for $\Ppk$) and of the Harnack inequality shows that  $\Ppmk$ are significant perturbations of the Laplacian.  In this context, the assumption $u\geq0$ in \eqref{eqbase-} is far to be equivalent to $u>0$. 

\smallskip
Let us consider the Dirichlet problem
\begin{equation}\label{Dirichlet}
\left\{
\begin{array}{rl}
u>0,\quad\Pmk(D^2u)+u^p=0 & \text{in $\Omega$}\\
u=0 &  \text{on $\partial\Omega$}.
\end{array}\right.
\end{equation}
The main result of this paper is the following
\begin{theorem}\label{existence}
Let $\Omega$ be a bounded domain of $\RN$ and let $k<N$ be a positive integer. 
\begin{itemize}
	\item[(1)] If $p\geq1$ then there are no subsolutions of \eqref{Dirichlet}.
	\item[(2)] If $p\in(0,1)$ and  $\Omega\in\CR$, then there exists a unique solution $U\in C(\overline\Omega)$ of \eqref{Dirichlet}.\\ Moreover for every $x_0\in\partial\Omega$ and any $q<\frac{1}{1-p}$ one has
\begin{equation}\label{hopf}
\lim_{{x\to x_0}}\frac{U(x)-U(x_0)}{|x-x_0|^q}=0.
\end{equation}
\end{itemize} 
\end{theorem}
\noindent
The conclusions of Theorem \ref{existence} are rather unusual and some comments are in order.

\smallskip
\noindent
The value $p=1$ sharply characterizes the range of $p$ for the existence and nonexistence of positive (sub)solutions of \eqref{Dirichlet}, so  playing the role of \lq\lq critical exponent\rq\rq\, for $\Pmk$, at least in the class $\CR$ of uniformly convex domains, see Definition \ref{convexity}.\\
It is quite surprising that the critical character of $p=1$  is exactly the same both for solutions and for subsolutions. Moreover it does not depend on $k<N$. This striking feature has also a dual counterpart in $\RN$: 
$p=1$ is a threshold  for the existence of positive (super)solutions of \eqref{eqbase-} in the whole $\RN$, as proved in \cite{BGL}.\\
Furthermore it appears there is no relationship between the standard Sobolev critical exponent $p^\star=\frac{N+2}{N-2}$ of the Laplacian $\Delta\equiv{\mathcal P}^-_N$ and that of $\Pmk$.\\
A further effect of the strong degeneracy of $\Pmk$ is expressed by \eqref{hopf} as a sort of \lq\lq anti-Hopf lemma\rq\rq. This is a consequence of the fact that $\Pmk$ is a first order operator, 
$$
\Pmk(D^2u)=k\frac{u'(r)}{r},
$$
 when acting on convex and radially decreasing functions $u=u(|x|)$,  namely $u'(r)\leq0$, $u''(r)\geq0$.

\medskip
In view of Theorem \ref{existence} it seems natural  to analyze the asymptotic profile for $p\to1$ of the solution $U=U_p$  of \eqref{Dirichlet}. Using only comparison arguments, we shall prove that 
$$
M_p:=\left\|U_p\right\|_{L^\infty(\Omega)}\to0\qquad\text{for $p\to1$}
$$
and, if $\Omega$ is the unit ball centered at the origin, that the rescaled functions
$$
\tilde U_p(x)=\frac{1}{M_p}U_p\left(M_p^{\frac{1-p}{2}}x\right)$$
  converge, as $p\to1$, to a radial solution of the \lq\lq limit equation\rq\rq 
$$\Pmk(D^2u)+u=0\qquad\text{in $\RN$}.$$
 This result continues to hold for small perturbations of the domain $\Omega$, as showed in the Proposition \ref{perturbations}. We  emphasize that such proof is performed without employing a priori estimates and compactness arguments, as for the Laplacian, see \cite{GS}, but only  using the explicit expression of $U_p$ in the unit ball and the comparison principle.

\medskip

Consider now the problem 
\begin{equation}\label{Dirichlet+}
\left\{
\begin{array}{rl}
\Ppk(D^2u)+u^p=0 & \text{in $\Omega$}\\
u=0 &  \text{on $\partial\Omega$}.
\end{array}\right.
\end{equation}
Here we don't need to assume $u>0$, since the strong minimum principle applies to every solution of $\Ppk(D^2u)\leq0$, see \cite{BGI}. We obtain the following 
\begin{theorem}\label{existence2}
Let $\Omega\in\CR$ and let $k<N$ be a positive integer. 
\begin{itemize}
	\item[(1)] If $p\in(0,1)$ then there exists a unique positive solution $V\in C(\overline\Omega)$ of \eqref{Dirichlet+}. \\ Moreover $V>U$ in $\Omega$, where $U$ is the solution of \eqref{Dirichlet} provided by Theorem \ref{existence}.
	\item[(2)] If $k=1$ and $p$ is any number larger than 1, then \eqref{Dirichlet+} admits a solution.
\end{itemize}
\end{theorem}
\noindent
The situation seems to be reversed with respect to the result concerning $\Pmk$ in the superlinear case, since at least for $k=1$ and $\Omega\in\CR$ there are nontrivial solutions for any $p>1$. In particular there are no \lq\lq critical values\rq\rq\, of $p$.
In order to explain such phenomena and describe how the proofs of Theorems \ref{existence}-\ref{existence2} are carried out, let us make a comparison  with the semilinear case $k=N$ and with the fully nonlinear uniformly elliptic counterpart.

\smallskip
\noindent
In the first case, both \eqref{eqbase-}-\eqref{eqbase+} reduce to  well known Lane-Emden equation
\begin{equation}\label{Laplacian}
\Delta u+u^p=0\qquad\text{in  $\Omega$}, \quad u=0\qquad\text{on $\partial\Omega$}.
\end{equation}
Since the problem \eqref{Laplacian} is in divergence form, then the existence issue of  positive solutions relies on the  minimization, in the Sobolev space $H^1_0(\Omega),$ of the associated energy functional 
$$
E(u)=\frac12\int_{\Omega}|Du|^2\,dx-\frac{1}{p+1}\int_{\Omega} |u|^{p+1}\,dx.
$$ 	
Hence variational methods provide the existence of positive solutions if $p<p^\star=\frac{N+2}{N-2}$. The bound on the exponent $p$ is related to the lack of compactness of the embedding $H^1_0(\Omega)\hookrightarrow L^{p^\star+1}(\Omega)$.  Moreover the classical Pohozaev's identity implies that such restriction on  $p$ is necessary if $\Omega$ is star shaped.  
 Of course, the operators $\Ppmk$ do not fit in the variational setting.\\
Another approach in establishing existence results is based on a priori estimates joined with topological degree arguments, as  in \cite{DFLN} (see also the survey paper \cite{L} and the references therein). This approach has been successfully applied  in the  fully nonlinear framework by Quaas and Sirakov \cite{QS} (see also \cite{FQ}, \cite{Q}). There the authors consider the Pucci's extremal operators ${\mathcal M}^\pm_{\lambda,\Lambda}$, the prototypes of fully nonlinear uniformly elliptic operators,  and they obtain a priori bounds by means of the blow-up technique \`a la Gidas-Spruck \cite{GS}. The success of these arguments rests first on the uniform ellipticity (regularity theory), then on  nonexistence results of Liouville type, which force the exponent $p$ to be  \lq\lq subcritical\rq\rq. The explicit expressions of the  critical exponents for entire supersolutions of ${\mathcal M}^\pm_{\lambda,\Lambda}(D^2u)+u^p=0$ are known, see \cite{CL}, and they reduce to the standard one $\frac{N}{N-2}$ when $\lambda=\Lambda=1$, i.e. when the Pucci's operators coincide with the Laplacian. A larger range of $p$ for the existence of radial solutions was obtained by Felmer-Quaas \cite{FQ2}, but in this case the explicit expressions for the critical exponents in terms of $\lambda,\Lambda,N$ are not known.  If $\Omega$ is an annulus, then the the existence of classical solutions of $F(x,D^2u)+u^p=0$ in $\Omega$, where $F$ is a radially invariant uniformly elliptic operator,  was obtained in \cite{GLP} for any $p>1$.

\smallskip

In the case of $\Ppmk$ the main obstruction in applying such arguments relies on the lack of compactness results in $C(\overline\Omega)$, whenever a priori estimates are available. If $k<N$ then the only known results concern the case $k=1$ and they are consequences  of H\"older regularity estimates, see \cite{BGI,OS}. For this reason we have avoided, as far as possible, the necessity of using compactness in proving our existence results. In the sublinear case, $p\in(0,1)$, for both $\Ppmk$ and any $k$ (not only $k=1$), we use a suitable change of variable, which goes back to \cite{BK} and converts \eqref{eqbase-}-\eqref{eqbase+} in proper equations for which the  Perron's method can be applied. Such approach has been already used  in the fully nonlinear setting in  \cite{CP}, but  there the uniformly ellipticity (ABP estimates, Harnack inequality etc.) is strongly used to get existence of solutions and, as mentioned before, such tools fail for $\Ppmk$. \\
On the other hand the implementation of the Perron's method  relies on the existence of barrier functions. To this aim one typically makes use of the uniformly  ellipticity. The  lack of uniformly ellipticity of $\Ppmk$ is offset  by a convexity assumption on $\Omega$, say $\Omega\in\CR$ in the sense of Definition \ref{convexity}. The convexity of $\Omega$ seems to be a natural assumption for the well-posedness  of the Dirichlet problem, see \cite{BGI,HL}.

In the superlinear case, $p>1$, the operators $\Pmk$ and $\Ppk$ produce totally different results. Such striking features have been already  investigated in \cite{BGL} for entire solutions. For $\Pmk$ the critical exponent that separates the existence and nonexistence range of positive (super)solutions is 1 for any $k$ and any dimension $N$, while $\Ppk$ acts similarly to the Laplacian in the $k$-dimensional space. Indeed for $k=1,2$ there are no entire positive supersolutions of $\Ppk(D^2u)+u^p=0$ for any $p$, while for $k>2$ the nonexistence result holds if, and only if, $p\leq\frac{k}{k-2}$. Moreover for $p\in\left(\frac{k}{k-2},\frac{k+2}{k-2}\right)$ there are no radial positive classical solutions, while such solutions exist when $p\geq\frac{k+2}{k-2}$. We observe that, differently from the Laplacian,  the failure of the strong maximum principle and the lack of symmetry results for $\Ppk$ does not allow to conclude that $\frac{k+2}{k-2}$ is the critical exponent for solutions. In particular the existence issue of nonradial solutions of $\Ppk(D^2u)+u^p=0$ in $\RN$ is an open problem.

\noindent
Another remarkable difference between $\Pmk$ and $\Ppk$ occurs at the eigenvalues level (see \cite{BGI}):  in any bounded domain $\Omega$, the operator $\Pmk$ does not admit positive eigenfunctions, since  $\Pmk(D^2\cdot)+\mu\cdot$ satisfies  the maximum principle for any  $\mu>0$. On the contrary the generalized principal eigenvalue \`a la Berestycki-Nirenberg-Varadhan for $\Ppk$, i.e. 
 $\sup\left\{\mu>0\,|\,\text{$\exists v>0$ s.t. $\Ppk(D^2v)+\mu v\leq0$ in $\Omega$}\right\}$, is finite and it corresponds (at least for $k=1$) to a positive eigenfunction.

The proof of  statement (2) of Theorem \ref{existence2} is based on a fixed point theorem and Leray-Schauder degree theory for compact operators. We follow the approach used by Quaas-Sirakov \cite{QS} in the case of Pucci's extremal uniformly elliptic operators. Such argument relies on compactness and this explain the reason for assuming $k=1$.   
There are no restriction on the exponent $p$ if $k=1$, since in this case the are no critical exponents  both in $\RN$ and  in the halfspace $\partial\RN_+$, as showed in \cite{BGL}. We stress that if the compactness of solutions is extended to $k>1$, then automatically the previous theorem would extend to these values of $k$ for $p$ below a corresponding critical exponent.

We conclude the Introduction by pointing out that the lack of information concerning   the critical exponent for $\Ppk$ and the uniqueness of  solutions of the corresponding equations, does not allow us to perform the asymptotic analysis as done in the case of $\Pmk$.

\bigskip

The paper is organized as follows: in Section \ref{Preliminaries} we recall some useful fact about the operators $\Ppmk$ and prove a comparison principle between subsolutions and (positive) supersolutions in the sublinear case $p\in(0,1)$; Section \ref{Sec3} is  devoted to the proof of Theorem \ref{existence} and to the asymptotic profile of the solution $U_p$ when $p\to1$; Theorem \ref{existence2} is proved in Section \ref{Sec4}; in the last section  we deal with a larger class of degenerate operators giving some extension of the previous results.

\section{Preliminaries}\label{Preliminaries}
We consider the linear space $\SN$ of real symmetric matrices of order $N$, $N\geq2$, endowed with the usual partial order: $X\leq Y$ in $\SN$ means that  $\left\langle X\zeta,\zeta\right\rangle\leq\left\langle Y\zeta,\zeta\right\rangle$ for any $\zeta\in\RN$, where $\left\langle\cdot,\cdot\right\rangle$ is the Euclidean inner product. Hereafter the eigenvalues of any $X\in\SN$ will be arranged in nondecreasing order:
$$
\lambda_1(N)\leq\ldots\leq\lambda_N(X).
$$ 
The norm of $X\in\SN$ is $\displaystyle\left\|X\right\|=\sup_{i=1,\ldots,N}\left\{|\lambda_i(X)|\right\}$, while $\tr(X)$ stands for the trace of $X$.\\ For $\zeta\in\RN$ the tensor product $X=\zeta\otimes\zeta$ is the symmetric matrix whose $i,j$-entry is $X_{ij}=\zeta_i\zeta_j$. In this case $\lambda_1(X)=0$, with multiplicity (at least) $N-1$ and eigenspace  orthogonal to $\zeta$,
 while $\lambda_N(X)=|\zeta|^2$ which is simple (if $\zeta\neq0$) and eigenspace spanned by $\zeta$.

%
\medskip

The operators $\Ppmk:\SN\mapsto\R$ are defined by
$$
\Pmk(X)=\lambda_1(X)+\ldots+\lambda_k(X)\qquad\text{and}\qquad\Ppk(X)=-\Pmk(-X).
$$
It is easy to see that both $\Pmk$ and $\Ppk$ are positively homogeneous operators of degree one, i.e $\Ppmk(tX)=t\,\Ppmk(X)$ for $t>0$, and $\Pmk(X)\leq\Ppk(X)$ for any $X\in\SN$.

Their  (degenerate) ellipticity  is expressed by the following monotonicity property:
\begin{equation}\label{ellipticity}
X\leq Y\quad\text{in $\SN$}\quad\Rightarrow\quad \Ppmk(X)\leq\Ppmk(Y).
\end{equation}
Condition \eqref{ellipticity} follows from  Courant's min-max representation formula for eigenvalues   $$ \lambda_i(X)=\min_{\dim V=i}\max_{\zeta\in V\backslash\left\{0\right\}}\frac{\left\langle X\zeta,\zeta\right\rangle}{|\zeta|^2}\qquad\text{for $i=1,\ldots,N$},$$
which tells in particular  that $\lambda_i(X)$ depends   monotonically  on $X$.\\
On the other hand, the strong notion of uniformly ellipticity, i.e. 
\begin{equation}\label{uniformellipticity}
\Ppmk(X)-\Ppmk(Y)\geq\lambda \tr(X-Y)\qquad\text{for some $\lambda>0$} 
\end{equation}
if $X\geq Y$, fails to be true. Actually it holds if, and only if, $k=N$.  When $k<N$ the degenerate traits of the operators $\Ppmk$ appear, for instance in the lack of strict ellipticity along any fixed direction $\zeta\in\RN$:
\begin{equation}\label{highdegeneracy}
\min_{X\in\SN}\left(\Ppmk(X+\zeta\otimes\zeta)-\Ppmk(X)\right)=0.
\end{equation}
Some properties concerning $\Ppmk$, such as sub/superadditivity or continuity in the matrix argument, can be easily deduced from the equivalent characterization of such operators as infimum and supremum of linear mappings acting on $\SN$ (see \cite[Lemma 8.1]{CLN}):
\begin{equation}\label{representation-}
\Pmk(X)=\inf\left\{\sum_{i=1}^k\left\langle X\zeta_i,\zeta_i\right\rangle\,|\,\,\text{$\zeta_i\in\RN$ and $\left\langle\zeta_i,\zeta_j\right\rangle=\delta_{ij}$ for $i,j=1,\ldots,k$}\right\}
\end{equation}
and
\begin{equation}\label{representation+}
\Ppk(X)=\sup\left\{\sum_{i=1}^k\left\langle X\zeta_i,\zeta_i\right\rangle\,|\,\,\text{$\zeta_i\in\RN$ and $\left\langle\zeta_i,\zeta_j\right\rangle=\delta_{ij}$ for $i,j=1,\ldots,k$}\right\}.
\end{equation}
In particular we deduce from \eqref{representation-}-\eqref{representation+} that $\Pmk$ is concave and $\Ppk$ is convex in $\SN$. Moreover the following inequalities are satisfied: for any $X,Y\in\SN$ 
\begin{equation}\label{subsuperadditivity}
\Pmk(X-Y)\leq\Ppmk(X)-\Ppmk(Y)\leq\Ppk(X-Y)
\end{equation}
and
\begin{equation}\label{continuity}
\left|\Ppmk(X)-\Ppmk(Y)\right|\leq k\left\|X-Y\right\|.
\end{equation}

\medskip
Let us look now at the solvability of the Dirichlet boundary value problem for $\Ppmk$ in a  bounded domain $\Omega$. Recall that, unless otherwise specified, by a \emph{solution} (resp. subsolution, supersolution)  we mean  \emph{continuous viscosity solution} (resp. upper semicontinuous viscosity subsolution, lower semicontinuous viscosity supersolution).

  The only ellipticity condition  \eqref{ellipticity} is sufficient to get a weak form of the maximum principle, hence the uniqueness of solutions.  The following maximum principle is a particular case of \cite[Proposition 4.1]{GV}.

\begin{proposition}\label{MP}
Let $u\in USC(\overline\Omega)$ be a subsolution of 
$$
\Ppk(D^2u)=f(x) \qquad \text{in $\Omega$},
$$
with $f\in C(\Omega)$ . Then
\begin{equation}\label{stima} 
\sup_\Omega u\leq\limsup_{x\to\partial\Omega}u(x)+C\left\|f^-\right\|_{L^\infty(\Omega)}
\end{equation}
where $C$ is a constant depending only on $k$ and ${\rm diam}\,\Omega$.
\end{proposition}

We focus on the existence issue.  It is well know that  the existence  of solutions of elliptic operators, satisfying  boundary conditions of Dirichlet type, is connected to the existence of barrier functions, which in turn depends on geometric properties of the boundary. In the case of $\Ppmk$, the barriers construction can be performed (see \cite[Section 2]{BGI}) in the case of uniformly convex domains:
\begin{definition}\label{convexity}
$\Omega\in\CR$ if, and only if, there exists $R>0$ and a subset $Y$ of $\RN$, depending on $\Omega$, such that
\begin{equation}\label{hulahoop}
\Omega=\bigcap_{y\in Y}B_R(y).
\end{equation}
\end{definition}
From the above definition it follows that  for any $x\in\partial\Omega$ there exists $y\in Y$, $y=y(x)$ and not necessarily unique, such that $B_R(y)\supseteq \Omega$ and $x\in\partial B_R(y)$. See also \cite[Chapter 14, Section 2]{GT}. We emphasize that, in general, the boundary of a domain $\Omega\in\CR$ is not a smooth surface, e.g. the intersection of a finite number of balls.\\
In this context (see \cite[Proposition 5.1]{BGI}): 

\begin{proposition}\label{exis}
Let $\Omega\in\CR$. If $f\in C(\Omega)$ is bounded, then the problems 
\begin{equation}\label{Dir}
\left\{
\begin{array}{rl}
\Ppmk(D^2u)=f(x) & \text{in $\Omega$}\\
u=0 \hspace{0.5cm}&  \text{on $\partial\Omega$}
\end{array}\right.
\end{equation}
have unique solutions.
\end{proposition}

We conclude this section by showing a comparison principle type result for
\begin{equation}\label{eqprel}
\Ppmk(D^2u)+u^p=0\quad\text{in $\Omega$}
\end{equation}
in the sublinear case $p\in(0,1)$. Here $\Omega$ is an arbitrary bounded domain of $\RN$. As already stated in the Introduction, the strong degeneracy of $\Ppmk$ is reflected on the validity of the maximum/minimum principles which are affected by lower order perturbations in the equations and they could be fail without further restrictions. As a matter of fact for any $x_0\in\Omega$ and $r$ small enough, let us consider the function
\begin{equation}\label{funz u}
u(x)=\left\{
\begin{array}{cl}
\left[\frac{1-p}{2k}(r^2-|x-x_0|^2)\right]^{\frac{1}{1-p}} & \text{if $x\in B_r(x_0)$}\\
0 & \text{if $x\in \overline\Omega\backslash B_r(x_0)$}.
\end{array}\right.
\end{equation}
It provides a nontrivial example of solution (see the proof of Theorem 1.2 in \cite{BGL}) of
$$
\Pmk(D^2u)+u^p=0\qquad\text{in  $\Omega$}, \quad u=0\qquad\text{on $\partial\Omega$}.
$$
 Hence the uniqueness of solutions with prescribed boundary data is violated. On the other hand if we restrict to the class of positive supersolutions, so excluding the functions  compactly supported in $\Omega$ as \eqref{funz u}, then the comparison between sub and supersolutions of \eqref{eqprel} holds true. The proof is carried out  by the change of variable $U(x)=\frac{1}{1-p}u^{1-p}(x)$, which goes back to \cite{BK}. We notice that   the only properties of $\Ppmk$  we shall use in the proof of comparison are the ellipticity and 1-homogeneity. For this reason and  for notation simplicity we shall detail the computations in the case of $\Pmk$, with obvious changes if $\Ppk$ is considered. \\
We first prove the following

\begin{lemma}\label{lemma}
If $u\in USC(\Omega)$ is a nonnegative subsolution ($u\in LSC(\Omega)$ supersolution) of \eqref{eqbase-} and $p\in(0,1)$, then the function $U$ defined by $U(x)=\frac{1}{1-p}u^{1-p}(x)$ is a nonnegative subsolution (supersolution) in $\Omega$ of
\begin{equation}\label{eq1lem}
\Pmk\left(UD^2U+\frac{p}{1-p}DU\otimes DU\right)+U=0.
\end{equation}
\end{lemma}
\begin{proof}
We treat the subsolution case, the other one being similar. Let $\varphi\in C^2(\Omega)$ be a test function touching $U$ from above at $x_0\in\Omega$, i.e.
$$
\max_{x\in B_\delta(x_0)}(U(x)-\varphi(x))=U(x_0)-\varphi(x_0)=0.
$$
If $U(x_0)=0$, then $x_0$ is a local minimum point for $\varphi$. Hence $D\varphi(x_0)=0$ and 
$$
\Pmk\left(\varphi(x_0)D^2\varphi(x_0)+\frac{p}{1-p}D\varphi(x_0)\otimes D\varphi(x_0)\right)+\varphi(x_0)=0.
$$
Otherwise $U(x_0)>0$ and the function $\psi(x)=(1-p)^{\frac{1}{1-p}}\varphi^{\frac{1}{1-p}}(x)$ touches $u$ from above at $x_0$. Since $u(x_0)>0$ then $\psi$ is $C^2$ in a neighborhood of $x_0$ (notice that if $p\geq\frac12$, then $\psi$ is $C^2$ even in the case $u(x_0)=0$). A straightforward computation yields
$$
D^2\psi(x_0)=(1-p)^{\frac{p}{1-p}}\varphi^{\frac{p}{1-p}-1}(x_0)\left(\varphi(x_0)D^2\varphi(x_0)+\frac{p}{1-p}D\varphi(x_0)\otimes D\varphi(x_0)\right)
$$
and, since $u$ is a subsolution of \eqref{eqbase-}, one has
\begin{equation*}
\begin{split}
0&\leq\psi^p(x_0)+\Pmk(D^2\psi(x_0))\\&=(1-p)^{\frac{p}{1-p}}\varphi^{\frac{p}{1-p}-1}(x_0)\left[\Pmk\left(\varphi(x_0)D^2\varphi(x_0)+\frac{p}{1-p}D\varphi(x_0)\otimes D\varphi(x_0)\right)+\varphi(x_0)\right].
\end{split}
\end{equation*}
Hence the conclusion follows from the fact that $\varphi(x_0)>0$.
\end{proof}

\begin{theorem}[\textbf{Comparison principle}]\label{CP}
 Let $u\in USC(\overline\Omega)$ and $v\in LSC(\overline\Omega)$ be respectively a sub and supersolution of \eqref{eqbase-} and let $p\in(0,1)$. If  $u\leq v$ on $\partial\Omega$ and $v$ is strictly positive in $\Omega$, then $u\leq v$ in $\overline\Omega$.
\end{theorem}
\begin{proof}
By contradiction we suppose that $\Omega^+=\left\{x\in\Omega:\,(u-v)(x)>0\right\}\not\equiv\emptyset$. Then for $\varepsilon>0$ fixed and small enough, we have that $\max_{\overline \Omega}(U-V_\varepsilon)>0$, where
$$ 
U(x)=\frac{1}{1-p}u^{1-p}(x)\quad\text{and}\quad V_\varepsilon(x)=(1+\varepsilon)V(x)=\frac{1+\varepsilon}{1-p}v^{1-p}(x)
$$ 
are respectively a nonnegative subsolution of \eqref{eq1lem} and a positive supersolution of 
\begin{equation}\label{eq2lem}
\Pmk\left(D^2V_\varepsilon+\frac{p}{1-p}\frac{DV_\varepsilon\otimes DV_\varepsilon}{V_\varepsilon}\right)+1\leq-\varepsilon.
\end{equation}
Doubling the variables, see \cite[Section 3]{CIL}, we consider for $\alpha>0$
\begin{equation}\label{eq3lem}
\begin{split}
\max_{\overline\Omega\times\overline\Omega}\left(U(x)-V_\varepsilon(y)-\frac\alpha2|x-y|^2\right)&=U(x_\alpha)-V_\varepsilon(y_\alpha)-\frac\alpha2|x_\alpha-y_\alpha|^2\\
&\geq\max_{\overline \Omega}(U(x)-V_\varepsilon(x))>0,
\end{split}
\end{equation}
with $(x_\alpha,y_\alpha)\in\Omega\times\Omega$ for $\alpha$ large. In view of \cite[Theorem 3.2]{CIL} there exists $X_\alpha$, $Y_\alpha\in\SN$ such that $X_\alpha\leq Y_\alpha$ and 
$$
(\alpha(x_\alpha-y_\alpha),X_\alpha)\in\overline J_\Omega^{2,+}U(x_\alpha), \quad (\alpha(x_\alpha-y_\alpha),Y_\alpha)\in\overline J_\Omega^{2,-}V_\varepsilon(y_\alpha).
$$
Using \eqref{eq1lem}, \eqref{eq2lem} and the fact that $U(x_\alpha)>V_\varepsilon(y_\alpha)$ as a consequence of \eqref{eq3lem} we get
\begin{equation*}
\begin{split}
0&\leq\Pmk\left(X_\alpha+\frac{p}{1-p}\frac{\alpha(x_\alpha-y_\alpha)\otimes\alpha(x_\alpha-y_\alpha)}{U(x_\alpha)}\right)+1\\
&\leq\Pmk\left(Y_\alpha+\frac{p}{1-p}\frac{\alpha(x_\alpha-y_\alpha)\otimes\alpha(x_\alpha-y_\alpha)}{V_\varepsilon(y_\alpha)}\right)+1\leq-\varepsilon,
\end{split}
\end{equation*}
a contradiction.
\end{proof}

\begin{remark}\label{rem}
\rm
We stress that Lemma \ref{lemma} and Theorem \ref{CP} can be restated for any degenerate elliptic operator $F=F(X)$, which is continuous and positively homogeneous of degree one. 
\end{remark}

\section{Positive solutions of $\Pmk(D^2u)+u^p=0$}\label{Sec3}

The comparison principle, Theorem \ref{CP}, provides the following a priori estimates.

\begin{proposition}\label{stime}
Let $u\in USC(\overline\Omega)$ be a subsolution of 
\begin{equation}\label{eq1}
\left\{
\begin{array}{rl}
\Pmk(D^2u)+u^p=0 & \text{in $\Omega$}\\
u=0 &  \text{on $\partial\Omega$}
\end{array}\right.
\end{equation}
in a bounded domain $\Omega$ such that $\Omega\subseteq B_R$. If $p\in(0,1)$ then
\begin{equation}\label{apriori}
\sup_\Omega u\leq \left(\frac{R^2}{2k}(1-p)\right)^{\frac{1}{1-p}}.
\end{equation}
If $p\geq1$ then $u\equiv0$.
\end{proposition}
\begin{remark}
\rm A further remarkable difference between the uniformly elliptic case $k=N$ and the degenerate one $k<N$, is the presence of ${(1-p)}^{\frac{1}{1-p}}$ in the estimate \eqref{apriori}. Such estimate shows that, for $p$ close to 1, all possible subsolutions of \eqref{eq1} are very small. This surprising feature trivially fails for the Laplacian. Indeed it is sufficient to consider the ball $B_r$ of radius $r$ such that the principal eigenvalue of $\Delta$ is $\lambda^+_1(B_r)=1$. Then consider the normalized positive eigenfunction $\phi$, in such a way $
\Delta \phi+\phi^p\geq\Delta\phi+\phi=0
$ for any $p\in(0,1)$ and $\phi=0$ on $\partial B_r$, but on the other hand $\sup_{B_r}\phi=1$. 
\end{remark}
\begin{proof}[Proof of Proposition \ref{stime}]
Let $p\in(0,1)$. Since the function $v(x)=\left[\frac{1-p}{2k}\left(R^2-|x|^2\right)\right]^{\frac{1}{1-p}}$ is a positive supersolution of $\Pmk(D^2v)+v^p=0$ in $\Omega$, in fact a classical solution, and $u\leq v$ on $\partial\Omega$, then Theorem \ref{CP} yields $u\leq v$ in $\Omega$. In particular
$$
\sup_\Omega u\leq \left(\frac{R^2}{2k}(1-p)\right)^{\frac{1}{1-p}}\,.
$$ 
Consider now the case $p\geq1$ and assume by contradiction that $u$ is a nontrivial subsolution of \eqref{eq1}. In particular $M=\sup_\Omega u\neq0$. Then the function $v(x)=\frac{1}{M}u\left(M^{\frac{1-p}{2}}x\right)$ is well defined in $\tilde\Omega=M^{\frac{p-1}{2}}\Omega$ and it is a subsolution of \eqref{eq1} in $\tilde\Omega$. Since $\sup_{\tilde\Omega}v=1$, then for any  for any $q<1$ 
\begin{equation*}
\Pmk(D^2v)+v^q(x)\geq\Pmk(D^2v)+v^p(x)\geq0\quad\text{in $\tilde\Omega$}.
\end{equation*}
By  estimate \eqref{apriori} 
\begin{equation}\label{e1}
\sup_{\tilde\Omega}v\leq \left(C(1-q)\right)^{\frac{1}{1-q}},
\end{equation}
 where $C$ is now a constant depending  on $k,\Omega,M$ and $p$. Sending $q\to1$ in \eqref{e1} we get $v\equiv0$. This is in contradiction to $\sup_{\tilde\Omega}v=1$. 
\end{proof}

\begin{proposition}\label{sub/supersol}
Let us consider the Dirichlet problem \eqref{eq1} in the sublinear case $p\in(0,1)$. Then:
\begin{itemize}
	\item[i)] there exists a positive subsolution $\underline u\in{\rm Lip}(\overline\Omega)$. In particular $\underline u=0$ on $\partial\Omega$;
	\item[ii)] if $\Omega\in\CR$, then there exists a positive supersolution $\overline u\in{\rm Lip}(\overline\Omega)$ such that $\overline u=0$ on $\partial\Omega$ and $\overline u\geq\underline u$ in $\overline\Omega$.
\end{itemize}
\end{proposition}
\begin{proof}
For any $z\in\Omega$, let $\delta_z={\rm dist}(z,\partial\Omega)$ be the distance function to the boundary of $\Omega$ and let $\underline v_z$ be the function defined by 
$$
\underline v_z(x)=\left\{
\begin{array}{cl}
\left[\frac{1-p}{2k}(\delta_z^2-|x-z|^2)\right]^{\frac{1}{1-p}} & \text{if $x\in B_{\delta_z}(z)$}\\
0 & \text{if $x\in \overline\Omega\backslash B_{\delta_z}(z)$}.
\end{array}\right.
$$
By construction $\underline v_z\in C^1(\overline\Omega)$, it is a solution of \eqref{eq1} and $\underline v_z\equiv0$ on $\partial\Omega$. The gradient of $\underline v_z$ can be bounded independently of $z\in\Omega$, say
$$
\left\|D\underline v_z\right\|_{L^\infty(\overline\Omega)}\leq C:=\left(\left(\frac{1-p}{1+p}\right)^{\frac{1+p}{2}}\frac{R^{1+p}}{k}\right)^{\frac{1}{1-p}}
$$
where $R$ is a positive constant such that $\Omega\subseteq B_R$. Hence the family $\left\{\underline v_z\right\}_{z\in\Omega}$ is equi-Lipschitz in $\overline\Omega$ and its supremum $\underline u(x)=\sup_{z\in\Omega}\underline v_z(x)$ is in turn Lipschitz continuous: for any $x_1,x_2\in\overline\Omega$ 
$$
\left|\underline u(x_1)-\underline u(x_2)\right|\leq\sup_{z\in\Omega}\left|\underline v_z(x_1)-\underline v_z(x_2)\right|\leq C|x_1-x_2|.
$$ 
From \cite[Lemma 4.2]{CIL}, $\underline u$ is a subsolution  of the equation in \eqref{eq1}. Moreover for any $z\in\Omega$ 
$$
\underline u(z)\geq\underline v_z(z)=\left(\frac{1-p}{2k}\delta_z^2\right)^{\frac{1}{1-p}},
$$
so the function $\underline u$ is positive in $\Omega$. On the boundary, if $x_0\in\partial\Omega$, then $\underline v_z(x_0)=0$ for any $z\in\Omega$, hence $\underline u(x_0)=0$. This concludes part \emph{i)}.\\
For \emph{ii)} we assume that  $\Omega=\bigcap_{y\in Y}B_R(y)$ where $Y\subseteq\RN$. For any $y\in Y$ we consider the positive solution of \eqref{eq1} in the ball $B_R(y)$, i.e. $v_y(x)=\left[\frac{1-p}{2k}(R^2-|x-y|^2)\right]^{\frac{1}{1-p}}$. Note that such solution is unique by comparison principle, Theorem \ref{CP}. As in part \emph{i)}, the infimum envelope $\overline u(x)=\inf_{y\in Y}v_y(x)$ defines a globally Lipschitz continuous function in $\overline\Omega$ which is a supersolution of the equation in \eqref{eq1}. Moreover if $x_0\in\partial\Omega$, then there exists at least one $y=y(x_0)\in Y$ such that $x_0\in\partial B_R(y)$. So $\overline u(x_0)\leq v_y(x_0)=0$ and $\overline u$ vanishes on the boundary $\partial\Omega$. Now we claim that $\overline u\geq\underline u$. We cannot directly invoke  Theorem \ref{CP} since $\overline u$, which is by definition the infimum of positive function, is a priori only nonnegative in $\Omega$. Fix $x_0\in\Omega$. Firstly  
note that 
\begin{equation}\label{eq2}
\underline u(x_0)=\sup\left\{\underline v_z(x_0):\,z\in\Omega\;\text{and}\;x_0\in B_{\delta_z}(z)\right\}
\end{equation}
since $\underline v_z(x_0)=0$ if $x_0\notin B_{\delta_z}(z)$. Now for any $y\in Y$ and $x_0\in B_{\delta_z}(z)\subseteq\Omega\subseteq B_R(y)$ one has
\begin{equation*}
\begin{split}
\left||x_0-y|^2-|x_0-z|^2\right|&=\left(|x_0-y|+|x_0-z|\right)\left||x_0-y|-|x_0-z|\right|\\
&\leq(R+\delta_z)|y-z|\\&\leq(R+\delta_z)(R-\delta_z)=R^2-\delta_z^2.
\end{split}
\end{equation*}
This implies that $\delta_z^2-|x_0-z|^2\leq R^2-|x_0-y|^2$, hence
$$
\underline v_z(x_0)\leq v_y(x_0).
$$
Since the previous inequality holds true for any $y\in Y$ and any $z\in\Omega$ such that $x_0\in B_{\delta_z}(z)$, we conclude by using \eqref{eq2}. 
\end{proof}

We are in position to prove Theorem \ref{existence}.

\begin{proof}[Proof of Theorem \ref{existence}]
The nonexistence of positive subsolutions in the superlinear case $p\geq1$ is a consequence of Proposition \ref{stime}.\\
Let us assume $p\in(0,1)$. In view Proposition \ref{sub/supersol}, the 
function
$$
U(x)=\sup_{u\in{\mathcal S}}u(x)
$$
where
$${\mathcal S}=\left\{u\in USC(\overline\Omega):\,\text{$u$ is a subsolution of \eqref{Dirichlet} and $\underline u\leq u\leq\overline u$ in $\overline\Omega$}\right\}$$  
is well defined in $\overline\Omega$. Since $\underline u>0$, the comparison principle, Theorem \ref{CP}, applies within ${\mathcal S}$. Then the Perron's method \cite[Theorem 4.1]{CIL}-\cite[Proposition II.1]{IL} guarantees that $U$ is the desired solution. For \eqref{hopf} fix any $x_0\in\partial\Omega$ and choose $y_0\in Y$, depending on $x_0$, such that $x_0\in\partial B_R(y_0)$. Using the definition of $\overline u$, i.e. $\overline u(x)=\inf_{y\in Y}v_y(x)$, one has
\begin{equation*}
\frac{U(x)-U(x_0)}{|x-x_0|^q}\leq\frac{\overline u(x)}{|x-x_0|^q}\leq\frac{v_{y_0}(x)}{|x-x_0|^q}\leq\left((1-p)\frac{R}{k}\right)^{\frac{1}{1-p}}|x-x_0|^{\frac{1}{1-p}-q}
\end{equation*}
and the conclusion follows.
\end{proof}

For any $p\in(0,1)$ let us denote by $U_p$ the unique solution of \eqref{Dirichlet} provided by Theorem \ref{existence}. We are interested in the asymptotic behaviour, as $p\to1$, of the rescaled functions
\begin{equation}\label{eq4}
\tilde U_p(x)=\frac{1}{M_p}U_p\left(M_p^{\frac{1-p}{2}}x\right)\qquad x\in \Omega_p\equiv M_p^{\frac{p-1}{2}}\Omega.
\end{equation}
Here $\displaystyle M_p=\max_{\overline \Omega}U_p$, so that $\displaystyle \max_{\overline \Omega_p}\tilde U_p=1$.\\
In the  case $\Omega\equiv B_1(0)$, the unit ball centered at the origin, we are going to show that $\tilde U_p$ converges locally uniformly to $\tilde U_1(x):=\exp(-\frac{|x|^2}{2k})$ which is a classical solution  of the limit equation 
\begin{equation}\label{eqRN}
\Pmk(D^2u)+u=0\qquad\text{in $\RN$}.
\end{equation}
At this stage it is worth to point out that  $\tilde U_1$ is the unique solution of \eqref{eqRN}, up to  translation and rescaling, within the class of $C^2$-radial functions. This fact is a consequence of the following lemma, see also \cite{BGL} for further properties concerning entire solutions of $\Ppmk(D^2u)+u^p=0$. 
\begin{lemma}\label{uniqradial}
Let $f\in C^1([0,+\infty))$ be  a nonnegative function such that $f'(t)\geq0$ for any $t\geq0$. If $u=u(r)$, $r=|x|$, is a nonnegative $C^2$-solution of 
$\Pmk(D^2u)+f(u)=0$ in $\RN$, then
\begin{equation}\label{eqrad1}
k\frac{u'(r)}{r}+f(u(r))=0\qquad\forall r>0.
\end{equation}  
\end{lemma}
\begin{proof}
We claim that
\begin{equation}\label{eqrad2}
\frac{u'(r)}{r}\leq u''(r)\qquad\forall r>0.
\end{equation}
From \eqref{eqrad2} we immediately obtain \eqref{eqrad1} using the definition of $\Pmk$. In order to prove \eqref{eqrad2} let us assume by contradiction that there exists $r_0>0$ such that the function
$$
v(r)=\left\{\begin{array}{cl}
\frac{u'(r)}{r}-u''(r) & \text{if $r>0$}\\
0 & \text{if $r=0$}
\end{array}\right.
$$
is positive at $r_0$. Let $\underline r=\inf\left\{\varrho\in(0,r_0):\;v>0\;\text{in}\;(\varrho,r_0)\right\}$. By continuity the function $v$ is positive on the interval $(\underline r,r_0)$ and $v(\underline r)=0$. Moreover in $(\underline r,r_0)$, the equation $\Pmk(D^2u)+f(u)=0$ can be written as
$$
v(r)=k\frac{u'(r)}{r}+f(u(r)).
$$
Differentiating the  equation above:
$$
v'(r)=k\frac{u''(r)}{r}-k\frac{u'(r)}{r^2}+f'(u(r))u'(r)\leq-\frac kr v,\qquad r\in (\underline r,r_0).
$$
In the last inequality we have used the assumption $f'\geq0$ and the fact that $u'\leq 0$, which is a direct consequence of the inequality $\Pmk(D^2u)\leq0$. Hence  $(r^kv)'\leq0$ in $(\underline r,r_0)$ and $r_0^kv(r_0)\leq\underline r^kv(\underline r)=0$. This contradicts the definition of $r_0$.
\end{proof}

Now we can go back to the asymptotic analysis of \eqref{eq4} in the case $\Omega=B_1$. This is a very simple computation since we know the explicit expression of $U_p$, namely $$U_p(x)=\left[\frac{1-p}{2k}(1-|x|^2)\right]^{\frac{1}{1-p}}.$$So  $\displaystyle M_p=U_p(0)=\left(\frac{1-p}{2k}\right)^{\frac{1}{1-p}}$ and 
$\displaystyle \Omega_p= B_{\sqrt{\frac{2k}{1-p}}}(0)$. When $p\to1$, then 
$
\Omega_p\to\RN
$
and
\begin{equation}
\lim_{p\to1}\tilde U_p(x)=\lim_{p\to1}\left(1-\frac{|x|^2}{2k}(1-p)\right)^{\frac{1}{1-p}}=\exp(-\frac{|x|^2}{2k}),
\end{equation}
the limit being uniform in compact sets of $\RN$. \\
A similar result continues to hold for domains $\Omega_p$ which are, in some sense, small perturbations of $B_{\sqrt{\frac{2k}{1-p}}}(0)$. This is precised in the following proposition.
\begin{proposition}\label{perturbations}
Let $\left\{\Omega_p\right\}_{p\in(0,1)}$ be a family of convex domains belonging to $\CR$, with $R=R(p)$ and $Y=Y(p)$. Denote by $U_p$ be the corresponding solutions of \eqref{Dirichlet} in $\Omega_p$. Let us assume that there exist $\bar x_p,\,\hat x_p\in\RN$ and $\bar\rho_p,\,\hat\rho_p>0$ such that
$$ B_{\sqrt{\frac{2k}{1-p}}-\bar\rho_p}(\bar x_p)\subseteq\Omega_p\subseteq B_{\sqrt{\frac{2k}{1-p}}+\hat\rho_p}(\hat x_p).$$
If 
\begin{equation}\label{eq1prop}
|\bar x_p|+|\hat x_p|\to0\;\quad\text{and}\;\quad\frac{\bar\rho_p+\hat\rho_p}{\sqrt{1-p}}\to0\qquad\text{as $p\to1$},
\end{equation}  
then $\displaystyle \lim_{p\to1}\tilde U_p(x)=\exp(-\frac{|x|^2}{2k})$ locally uniformly in $\RN$.
\end{proposition}
\begin{proof}
Let 
\begin{equation*}
\begin{split}
u_p(x)&={\left[\frac{1-p}{2k}\left(\left(\sqrt{\frac{2k}{1-p}}-\bar\rho_p\right)^2-|x-\bar x_p|^2\right)\right]}^{\frac{1}{1-p}}\\
&={\left[1+\frac{1-p}{2k}\left(\bar\rho_p^2-|\bar x_p|^2+2\left\langle x,\bar x_p\right\rangle-|x|^2\right)-2\bar\rho_p\sqrt{\frac{1-p}{2k}}\right]}^{\frac{1}{1-p}}
\end{split}
\end{equation*}
and 
\begin{equation*}
\begin{split}
v_p(x)&={\left[\frac{1-p}{2k}\left(\left(\sqrt{\frac{2k}{1-p}}+\hat\rho_p\right)^2-|x-\hat x_p|^2\right)\right]}^{\frac{1}{1-p}}\\
&={\left[1+\frac{1-p}{2k}\left(\hat\rho_p^2-|\hat x_p|^2+2\left\langle x,\hat x_p\right\rangle-|x|^2\right)+2\hat\rho_p\sqrt{\frac{1-p}{2k}}\right]}^{\frac{1}{1-p}}
\end{split}
\end{equation*}
be  the solutions of \eqref{Dirichlet} in  $\displaystyle  B_{\sqrt{\frac{2k}{1-p}}-\bar\rho_p}(\bar x_p)$ and $\displaystyle B_{\sqrt{\frac{2k}{1-p}}+\hat\rho_p}(\hat x_p)$ respectively. \\
Fix a compact subset $K$ of $\RN$. Since $\displaystyle  B_{\sqrt{\frac{2k}{1-p}}-\bar\rho_p}(\bar x_p)$ tends to $\RN$, then for $p$ close to 1 the comparison principle, Theorem \ref{CP}, yields
$$
u_p(x)\leq U_p(x)\leq v_p(x)\qquad x\in K.
$$  
Using the assumption \eqref{eq1prop}, it is easy to check that both $u_p$ and $v_p$ converge to $\tilde U_1(x)=\exp(-\frac{|x|^2}{2k})$ uniformly in $K$, from which the conclusion follows.

\end{proof}

\section{Positive solutions of $\Ppk(D^2u)+u^p=0$}\label{Sec4}

We deal with the existence of positive solutions of
\begin{equation}\label{DP for Ppk}
\left\{
\begin{array}{rl}
\Ppk(D^2u)+u^p=0 & \text{in $\Omega$}\\
u=0 &  \text{on $\partial\Omega$}.
\end{array}\right.
\end{equation}
As a consequence of the strong minimum  principle (see \cite[Remark 2.6]{BGI}), every supersolution $u$ of \eqref{DP for Ppk} satisfies: either $u\equiv0$ or $u>0$ in $\Omega$, independently of $p$.\\
First we consider the sublinear case $p\in(0,1)$, giving the proof of Theorem \ref{existence2} part (1).

\begin{proof}[Proof of Theorem \ref{existence2}-(1)]
Using the inequality $\Pmk(X)\leq\Ppk(X)$, which holds for any matrix $X\in\SN$, we infer that $U$, the solution of \eqref{Dirichlet},   is a  subsolution of  \eqref{DP for Ppk} vanishing on the boundary $\partial\Omega$.\\
Let $\Omega=\bigcap_{y\in Y}B_R(y)$ and consider for any $y\in Y$  the function $v_y$  defined by 
$$
v_{y}(x)=\tau(R^2-|x-y|^2)\qquad x\in\overline\Omega,
$$
where $\displaystyle\tau=\left(\frac{R^{2p}}{2k}\right)^{\frac{1}{1-p}}$. Since $D^2v_y(x)=-2\tau I$, where $I\in\SN$ is the identity matrix, then for any $x\in\Omega$ we obtain
\begin{equation*}
\Ppk(D^2v_y(x))+v_y^p(x)=-2k\tau+\tau^p(R^2-|x-y|^2)^p\leq\tau^p(R^{2p}-2k\tau^{1-p})=0.
\end{equation*}
Moreover
$$
\sup_{y\in Y}\left\|Dv_y\right\|_{L^\infty(\overline \Omega)}\leq2\tau R.
$$
As in the proof of Proposition \ref{sub/supersol}-ii), the function $\overline  u(x)=\inf_{y\in Y}v_y(x)\in{\rm Lip}(\overline\Omega)$ is a supersolution of \eqref{DP for Ppk} such that $\overline  u\equiv0$ on $\partial\Omega$. As far as the positivity of $\overline u$ is concerned, fix $x_0\in\Omega$ and let $\delta_{x_0}=\text{dist$(x_0,\partial\Omega)$}$. Since $B_{\delta_{x_0}}(x_0)\subseteq B_R(y)$ for any $y\in Y$, then $|y-x_0|\leq R-\delta_{x_0}$ and 
$$
v_y(x_0)=\tau(R+|x_0-y|)(R-|x_0-y|)\geq\tau R\delta_{x_0}.
$$
Taking the infimum over $y\in Y$ we get $\overline u(x_0)\geq\tau R\delta_{x_0}$. \\
In view of comparison principle, Theorem \ref{CP} and Remark \ref{rem}, we obtain the inequality $U\leq\overline u$ in $\overline\Omega$. Hence the Perron's solution
$$
V(x)=\sup\left\{u(x):\,\text{$u$ is a subsolution of \eqref{DP for Ppk} and $U\leq u\leq\overline u$ in $\overline\Omega$}\right\}
$$
provides the unique positive solution of \eqref{DP for Ppk}.\\
Since $U\leq V$ in $\Omega$ by definition, then the difference $V-U$ is nonnegative and  satisfies in the viscosity sense (see \cite[Lemma 3.1]{GV})
$$
0=V^p(x)-U^p(x)+\Ppk(D^2V(x))-\Pmk(D^2U(x))\geq\Ppk(D^2(V-U)(x))\qquad x\in\Omega.
$$
By the strong minimum principle either $U\equiv V$ or $U<V$. We argue by contradiction by assuming that $U\equiv V$. Take $z_0\in\Omega$  and let $x_0\in\partial\Omega$ such that $\delta_{z_0}=\text{dist$(z_0,\partial\Omega)$}=|x_0-z_0|$. Since we are assuming that $U$ is also a positive solution of \eqref{DP for Ppk}, then we can use the Hopf boundary lemma (which holds true for supersolution of $\Ppk$, see \cite{BGI}). Hence if $\nu=\frac{x_0-z_0}{|x_0-z_0|}$ denotes the outer normal to the ball
$B_{\delta_{z_0}}(z_0)$ at $x_0$, then
$$
\liminf_{t\to0^+}\frac{U(x_0-t\nu)-U(x_0)}{t}>0.
$$
But this in in contradiction with \eqref{hopf} and this shows that $U<V$ in $\Omega$. 
\end{proof}

Henceforth we shall assume $p>1$. In this case, the proof of the existence of a positive solution of \eqref{DP for Ppk} consists in the application of the degree theory for compact operators in cones, combined with 
 appropriate a priori bounds. We shall use in particular the following fixed point theorem.
\begin{theorem}[\textbf{\cite[Proposition 2.1 and Remark 2.1]{DFLN}}]\label{fixedpoint}
Let ${\mathcal C}$ be a cone in the Banach space $({\mathcal X},\left\|\cdot\right\|)$ and $\Phi:{\mathcal C}\mapsto{\mathcal C}$ a compact map such that $\Phi(0)=0$. Assume that there exist a compact map $F:\overline   B_R\times[0,\infty)\mapsto{\mathcal C}$, where $B_R=\left\{x\in{\mathcal C\,|\,\left\|x\right\|< R}\right\}$, and numbers $0<r<R$, $T>0$ such that
\begin{itemize}
	\item [(i)]  $x\neq t\Phi(x)$ for $0\leq t\leq1$ and $\left\|x\right\|=r$;
	\item [(ii)] $F(x,0)=\Phi(x)$  and $F(x,t)\neq x$ for $\left\|x\right\|=R$ and $0\leq t<\infty$;
	\item [(iii)] $F(x,t)=x$ has no solution $x\in\overline B_R$ for $t\geq T$.
\end{itemize}
Then $\Phi$ has a fixed point $x\in{\mathcal C}$ such that $r<\left\|x\right\|<R$.
\end{theorem}

In our case let  ${\mathcal C}$ be the closed cone in the Banach space ${\mathcal X}=C(\overline\Omega)$, defined by
$${\mathcal C}=\left\{v\in C(\overline \Omega) \,|\,\text{$v\geq0$ in $\overline\Omega$ and $v=0$ on $\overline\Omega$}\right\}.$$  
We consider the map $\Phi_k:{\mathcal C}\mapsto {\mathcal C}$, where for $v\in{\mathcal C}$, $u(x)=\Phi_k(v)(x)$ is the  unique solution of
\begin{equation}\label{DPT}
\left\{
\begin{array}{rl}
\Ppk(D^2u)+v^p=0 & \text{in $\Omega$}\\
u=0 &  \text{on $\partial\Omega$}.
\end{array}\right.
\end{equation}
The existence issue of a positive solution of \eqref{DP for Ppk} is equivalent to find  a fixed point of $\Phi_k$.

Using Propositions \ref{MP} and \ref{exis}, the map  $\Phi_k$ is well defined and such that $\Phi_k(0)=0$.  Moreover if $v_n\to v$ in ${\mathcal C}$ and $u_n=\Phi_k(v_n)$, then
\begin{equation*}
\left\{
\begin{array}{rl}
\Ppk(D^2|u_n-u|)\geq-\left|v^p_n(x)-v^p(x)\right| & \text{in $\Omega$}\\
u_n-u=0 &  \text{on $\partial\Omega$}.
\end{array}\right.
\end{equation*}
Using the estimate \eqref{stima} we obtain
$$
\left\|\Phi_k(v_n)-\Phi_k(v)\right\|_{L^\infty(\overline\Omega)}\leq C\left\|v^p_n-v^p\right\|_{L^\infty(\overline\Omega)},
$$
from which the continuity of $\Phi_k$ follows. \\
In order to have the compactness of $\Phi_k$ we shall assume from now on that $k=1$. Setting $\Phi_1=\Phi$, in view of \cite[Proposition 3.1]{BGI},  $\Phi$ maps bounded subsets of ${\mathcal C}$ in bounded subsets of ${\rm Lip}(\overline\Omega)$, which are precompact in ${\mathcal C}$ by Ascoli theorem.\\
Now we define  $F(v,t)$, for any $v\in{\mathcal C}$ and any $t\geq0$, as the unique solution $u(x)$ of 
\begin{equation}\label{DPT2}
\left\{
\begin{array}{rl}
\Ppo(D^2u)+{(v+t)}^p=0 & \text{in $\Omega$}\\
u=0 &  \text{on $\partial\Omega$}.
\end{array}\right.
\end{equation}
By the same argument as before, $F$ is a compact operator in $\mathcal C\times[0,\infty)$. Moreover by the definition, $F(\cdot,0)\equiv\Phi(\cdot)$.\\
We  conclude the proof of Theorem \ref{existence2}.

\begin{theorem}\label{theoremP1}
Let $\Omega\in\CR$ and let $p>1$. Then there exists a positive solution of
\begin{equation}\label{DP for Pp1}
\left\{
\begin{array}{rl}
\Ppo(D^2u)+u^p=0 & \text{in $\Omega$}\\
u=0 &  \text{on $\partial\Omega$}.
\end{array}\right.
\end{equation}
\end{theorem}
\begin{proof}
We are going to prove that the assumptions of Theorem \ref{fixedpoint} are satisfied. \
We start by (i). For this let us  consider the positive first eigenvalue $\mu^+_1$ of $\Ppo$ associated to a positive eigenfunction,  defined by
\begin{equation}\label{princpaleigenvalue}
\mu^+_1=\sup\left\{\mu>0\,|\,\text{$\exists v\in LSC(\Omega)$, $v>0$ and $\Ppo(D^2v)+\mu v\leq0$ in $\Omega$}\right\}.
\end{equation}
It gives a threshold for the validity of the maximum principle (see \cite{BGI}), i.e. for  any $\mu<\mu^+_1$
\begin{equation}\label{MaxPrin}
\text{$\Ppo(D^2u)+\mu u\geq0$ in $\Omega$, \;$u\leq0$ on $\partial\Omega\;\;\Longrightarrow\;\;u\leq0$ in $\Omega$.}
\end{equation}
We show that there exists $r>0$ small enough such that there are no positive subsolutions of
\begin{equation*}
\left\{
\begin{array}{rl}
\Ppo(D^2v)+tv^p=0 & \text{in $\Omega$}\\
v=0 &  \text{on $\partial\Omega$},
\end{array}\right.
\end{equation*}
$t\in[0,1]$, with $\left\|v\right\|_{L^\infty(\Omega)}\leq r$. This in particular implies that the equation  $v=t\Phi(v)$ does not admit positive solutions $v$ with $\left\|v\right\|_{L^\infty(\Omega)}=r$, i.e. condition (i). We argue by contradiction by assuming that there exists a sequence $t_n\in[0,1]$ and  a sequence of positive function $v_n\in USC(\overline\Omega)$ such that $\left\|v_n\right\|_{L^\infty(\Omega)}\to0$ and
$$
\Ppo(D^2v_n)+t_n v_n^p\geq0\quad\text{in $\Omega$},\qquad v_n=0\quad\text{on $\partial\Omega$}.
$$
For $\varepsilon<\mu_1^+$, we can pick $n$ large enough so that $v_n>0$ in $\Omega$ and
$$
0\leq\Ppo(D^2v_n)+t_n v_n^p\leq\Ppo(D^2v_n)+t_n \left\|v_n\right\|_{L^\infty(\Omega)}^{p-1}v_n\leq \Ppo(D^2v_n)+\varepsilon v_n,
$$
but this contradicts \eqref{MaxPrin}.

 Now we focus on condition (iii). Let $T$ be such that  $T>\left(\mu_1^+\right)^\frac{1}{p-1}$. If, for some $t\geq T$, $v$ would a nonnegative supersolution of $\Ppo(D^2v)+(v+t)^p=0$, then
$$
\Ppo(D^2(v+t))+T^{p-1}(v+t)\leq\Ppo(D^2v)+(v+t)^p\leq0
$$
and this contradicts the definition \eqref{princpaleigenvalue} of $\mu_1^+$. This means that for every nonnegative supersolution of  $\Ppo(D^2v)+(v+t)^p=0$ in $\Omega$, then $0\leq t< T$. This in particular implies (iii) and condition (ii) for $t\geq T$.

It remains to show condition (ii) for $t\in[0,T)$. This will be accomplished by showing that for $R$ large enough, then all possible positive solutions of 
\begin{equation*}
\left\{
\begin{array}{rl}
\Ppo(D^2u)+{(u+t)}^p=0 & \text{in $\Omega$}\\
u=0 &  \text{on $\partial\Omega$},
\end{array}\right.
\end{equation*}
with $0\leq t<T$, satisfy the a priori bound $\left\|u\right\|_{L^\infty(\Omega)}<R$. If this is not the case, then there exists $\left\{(u_n,t_n)\right\}\subset{\mathcal C}\times[0,T)$ such that  $$\Ppo(D^2u_n)+(u_n+t_n)^p=0\;\; \text{in}\;\; \Omega\quad \text{and}\quad\left\|u_n\right\|_{L^\infty(\Omega)}\to\infty.$$ Set $M_n=\max_{\overline\Omega}u_n=u_n(x_n)$, $x_n\in\Omega$ and 
$$
\tilde u_n(x)=\frac{1}{M_n}u_n\left(x_n+M_n^{\frac{1-p}{2}}x\right)\qquad x\in \Omega_n\equiv{M_n}^{\frac{p-1}{2}}(\Omega-\left\{x_n\right\}).
$$
By scaling invariance, $\tilde u_n$ is a solution of 
\begin{equation}\label{eqinfty}
\Ppo(D^2\tilde u_n)=-\left(\tilde
u_n+\frac{t_n}{M_n}\right)^p\qquad\text{in \;$\Omega_n$},
\end{equation}
with right hand side uniformly bounded in $L^\infty$. It is standard to see that $\Omega_n$ tends as $n\to\infty$ to $\RN$ or to the half spaces $\partial\RN_+$ (up to a orthogonal transformation). Using the regularity result of \cite[Proposition 3.1]{BGI} and a diagonal argument, there exists a subsequence of $\tilde u_n$, still denoted by $\tilde u_n$, converging locally uniformly  to a function $U$. Passing to the limit into the equation \eqref{eqinfty} we get that $U$ is a solution of 
$$
\Ppo(D^2U)+U^p=0\qquad\text{in $\RN$}
$$
or of 
\begin{equation*}
\left\{
\begin{array}{rl}
\Ppo(D^2U)+{U}^p=0 & \text{in $\RN_+$}\\
U=0 &  \text{on $\partial\RN_+$}.
\end{array}\right.
\end{equation*}
On the other hand the Liouville type results \cite[Theorems 1.1,1.3,1.5]{BGL} imply in both cases that $U\equiv0$, but this is in contradiction to $\left\|U\right\|_{L^\infty}=1$.
\end{proof}

\section{Some generalization}\label{Generalizations}

In this section we  deal with the class of one homogeneous elliptic operators for which $\Pmk$ and $\Ppk$ are respectively the minimal and maximal ones, with the aim to extend the existence results in the sublinear case to this larger class of operators. We shall also give some generalization in the superlinear case.\\
Therefore we consider the boundary value problem
\begin{equation}\label{4eq1}
\left\{
\begin{array}{rl}
u>0,\quad F(D^2u)+a(x)u^p=0 & \text{in $\Omega$}\\
u=0 & \text{on $\partial\Omega$}
\end{array}\right.
\end{equation}
under the following assumptions:
\begin{itemize}
	\item[(H1)] $F:\SN\mapsto\R$ is a continuous degenerate elliptic operator, positively homogeneous of degree 1  such that 
	\begin{equation}\label{4eq2}
	\Pmk(X)\leq F(X)\leq\Ppk(X)\qquad\forall X,Y\in\SN
	\end{equation}
	for some  $k\in[1,N]$;
	\item[(H2)] $a\in C(\Omega)$ is bounded between two positive constants. 
\end{itemize}

\medskip

\noindent
\textbf{Examples of $F$} 
 
\medskip
\noindent
\textbf{1.} Convex combinations of eigenvalues:
$$
F(X)=\sum_{i=1}^N\alpha_i\lambda_i(X)\qquad\text{with}\qquad \sum_{i=1}^N\alpha_i=1\quad\text{and}\quad \alpha_i\geq0
$$
satisfy (H1) with $k=1$. As particular cases:
\begin{center}
\begin{enumerate}
	\item[\textbf{i)}] $\displaystyle\alpha_i=\frac{1}{N}\quad\forall i=1,\ldots,N \;\Longrightarrow\; F(D^2u)=\frac1N\Delta u;$
	\item[\textbf{ii)}] $\displaystyle\alpha_i=1,\;\alpha_j=0\;\;\text{for}\;\;j\neq i\;\Longrightarrow\;  F(D^2u)=\lambda_i(D^2u).$
\end{enumerate}
\end{center}
\textbf{2.} $k$-sums of eigenvalues (not necessarily consecutive eigenvalues):
\begin{equation}\label{4eq3}
F(X)=\sum_{i=1}^k\lambda_{j_i}(X)
\end{equation}
where $1\leq j_1<\ldots<j_k\leq N$ are integer numbers.\\
Note that the extremal operators $\Pmk$ and $\Ppk$ correspond to $j_i=i$ and $j_i=N-k+i$ respectively. \\
If $k=N-1$, then \eqref{4eq3} can be written as
	$$
	F(X)=\tr(X)-\lambda_i(X)
	$$
	for a certain $i\in\left\{1,\ldots,N\right\}$, so that $F(D^2u)$ is some sort or \lq\lq quasi-Laplacian\rq\rq.

\medskip
\noindent\textbf{3.} Sup/inf operators:	
$$
F(X)=\sup_{\alpha}F_\alpha(X), \qquad F(X)=\inf_\alpha F_\alpha(X)
$$
where $\left\{F_\alpha\right\}_{\alpha}$ is a one-parameter family of elliptic operators satisfying \eqref{4eq2}. More general one could consider families of operators depending on two or more parameters.\\ Just to give an explicit example, consider
$$
F(X)=\max\left\{\lambda_1(X)+\lambda_4(X),\lambda_2(X)+\lambda_3(X)\right\}\qquad X\in{\mathcal S}^4,
$$
which fulfill the assumption (H1) with $k=2$.

\begin{theorem}
Let $\Omega\in\CR$ and assume (H1)-(H2). If $p\in(0,1)$ then there exist a unique solution of \eqref{4eq1}.
\end{theorem}
\begin{proof}[Sketch of the proof.] Let $U$ and $V$ be respectively the functions obtained in Theorems \ref{existence}-\ref{existence2}. By the assumption (H2)  
$$
0<\underline a:=\inf_{x\in\Omega} a(x)\leq\overline a:=\sup_{x\in\Omega}a(x)<\infty.
$$
Hence, using \eqref{4eq2}, we get that $u(x)=\underline a^{\frac{1}{1-p}}U$ and $v(x)=\overline a^{\frac{1}{1-p}}V$ are respectively  positive sub and supersolution of  $$F(D^2u)+a(x)u^p=0\qquad\text{in $\Omega$}$$ which are null on $\partial\Omega$. Since the comparison principle also applies  for \eqref{4eq1} (see Remark \ref{rem} and use the continuity of $a(x)$), then the existence of the unique positive solution is a consequence of Perron's method.
\end{proof}

\begin{remark}
\rm 
In the particular case $F(X)=\frac{\tr X}{N}$ and $a(x)=\frac1N$, then the previous theorem provides the existence and uniqueness of a unique positive viscosity solution of $\Delta u+u^p=0$. A variational approach for such problem (with more general zero order terms) has been addressed in Brezis-Oswald
 \cite{BO}
\end{remark}

 One of the crucial difference between $\Pmk$ and $\Ppk$ is expressed by the reversed Hopf property \eqref{hopf} of Theorem \ref{existence}. It relies on the fact that, for 
\begin{equation*}
v(x)=\left[\frac{1-p}{2k}(R^2-|x-x_0|^2)\right]^{\frac{1}{1-p}}\qquad x_0\in\partial\Omega,
\end{equation*}
 the largest eigenvalue $\lambda_N(D^2v)$ of the Hessian is neglected  by the operator $\Pmk$ by definition, conversely it appears in the expression of $\Ppk(D^2v)$. This  implies that $v$ is a solution of $\Pmk(D^2v)+v^p=0$ in $\Omega$, while it is not for $\Ppk$.\\
In this spirit, using a comparison argument, it is easy to see that the unique positive solution $u$ of \eqref{4eq1} still satisfies 
$$
\lim_{x\to x_0}\frac{u(x)-u(x_0)}{|x-x_0|^q}=0\qquad\text{for $q<\frac{1}{1-p}$}
$$
provided the function $v$ (multiplied by a  positive constant depending on $F$ and $a(x)$) is a solution of the inequality
\begin{equation}\label{4eq4}
F(D^2v)+a(x)v^p\leq0\qquad\text{in $\Omega$}.
\end{equation}
See Remark \ref{remark end} below for some examples.

\bigskip

We now consider $p\geq1$. As showed in the  Section \ref{Sec3},  the nonexistence of subsolutions of \eqref{4eq1} in the case $F\equiv \Pmk$ is a consequence of the a priori bound \eqref{apriori}. An alternative way to prove this fact is the following: first observe that for any nonnegative subsolution of $\Pmk(D^2u)+a(x)u^p=0$ one has
$$
\Pmk(D^2u)+\left\|au^{p-1}\right\|_{L^\infty(\Omega)}u\geq0\qquad\text{in $\Omega$};
$$
then in \cite[Proposition 4.3]{BGI} it has been proved that 
$$
\sup\left\{\mu>0\,|\,\text{$\exists v\in LSC(\Omega)$, $v>0$ and $\Pmk(D^2v)+\mu v\leq0$ in $\Omega$}\right\}=\infty.
$$
This implies that for any $\mu\in\R$ the operator $\Pmk(D^2\cdot)+\mu\cdot$ satisfies the  maximum principle (\cite[Theorem 4.1]{BGI}), hence
$$
u\leq0\quad\text{on $\partial\Omega$}\quad\Longrightarrow\quad u\equiv 0\quad\text{in $\overline\Omega$}.
$$
More in general
\begin{proposition}\label{4prop1}
Let $\Omega$ be a bounded domain and assume (H1)-(H2). If
\begin{equation}\label{4eq5}
\sup\left\{\mu>0\,|\,\text{$\exists v\in LSC(\Omega)$, $v>0$ and $F(D^2v)+\mu v\leq0$ in $\Omega$}\right\}=\infty
\end{equation}
then there are no subsolutions of \eqref{4eq1}.
\end{proposition}
\begin{remark}\label{remark end}
\rm
Conditions \eqref{4eq4}-\eqref{4eq5} are satisfied, for instance, by the convex combinations of Example 1 in the case $\alpha_N=0$, or by the partial $k$-sums of Example 2 provided $j_k\leq N-1$. For \eqref{4eq5} it is sufficient to consider  $v(x)=(R^2-|x|^2)^\gamma$, $\gamma>1$, as in the case of $\Pmk$, see \cite[Proposition 4.3]{BGI}. 
\end{remark}
Finally we point out that the proof of Theorem \ref{theoremP1} strongly depends on the precise structure of $\Ppo$, in particular on the compactness of the map $\Phi$ defined in \eqref{DPT}. The lack of regularity results for $\Ppk$, $k>1$, seems to be a real obstruction in the existence issue of nontrivial solutions,  a fortiori for more general nonlinearities $F$ satisfying (H1) and which are not covered by Proposition \ref{4prop1}.   
	
	\bigskip
	
	\noindent
	\textbf{Acknowledgments.} The author is deeply indebted to  Isabeau Birindelli for her helpful advice and for the interesting discussions they had together.\\
	The author was partially supported by GNAMPA-INDAM.


\begin{thebibliography}{40}
\bibitem{AS} L. Ambrosio, H.M. Soner, \emph{Level set approach to mean curvature flow in arbitrary codimension}, J. Differ. Geom. 43 (1996) 693-737.
\bibitem{AGV} M.E. Amendola, G. Galise, A. Vitolo, \emph{Riesz capacity, maximum principle, and removable sets of fully nonlinear second-order elliptic operators}, Differ. Integral Equ. 26 (2013) 845-866.
\bibitem{BNV} H. Berestycki, L. Nirenberg, S. Varadhan, \emph{The principle eigenvalue and maximum principle for second order elliptic operators in general domains}, Commun. Pure Appl. Math. 47(1) (1994) 47-92.
\bibitem{BGI} I. Birindelli, G. Galise, I. Ishii, \emph{A family of degenerate elliptic operators: Maximum principle and its consequences}, Ann. Inst. H. Poincar\'e Anal. Non Lin\'eaire, (2018) in press.
\bibitem{BGL} I. Birindelli, G. Galise, F. Leoni,  \emph{Liouville theorems for a family of very degenerate elliptic nonlinear operators}, Nonlinear Anal. 161 (2017) 198-211.
\bibitem{BK} H. Brezis, S. Kamin, \emph{Sublinear elliptic equations in $\R^n$}, Manuscripta Math. 74 (1) (1992) 87-106.
\bibitem{BO} H. Brezis, L. Oswald, \emph{Remarks on sublinear elliptic equations}, Nonlinear Anal. 10 (1986) 55-64.
\bibitem{CC} L.A. Caffarelli and X. Cabr\'e , \emph{Fully Nonlinear Elliptic Equations}, Colloquium Publications 43, American Mathematical Society, Providence, RI, 1995.
\bibitem{CLN} L. Caffarelli, Y.Y. Li, L. Nirenberg, \emph{Some remarks on singular solutions of nonlinear elliptic equations. I}, J. Fixed Point Theory Appl. 5 (2009) 353-395.
\bibitem{CLN2} L. Caffarelli, Y.Y. Li, L. Nirenberg, \emph{Some remarks on singular solutions of nonlinear elliptic equations III: viscosity solutions including parabolic operators}, Commun. Pure Appl. Math. 66 (2013) 109-143.
\bibitem{CDLV} I. Capuzzo Dolcetta, F. Leoni, A. Vitolo, \emph{On the inequality $F(x, D^2u) \geq f(u) +g(u)|Du|^q$}, Math. Ann. 365(1–2) (2016) 423-448.
\bibitem{CP} F. Charro, I. Peral,  \emph{Zero order perturbations to fully nonlinear equations: comparison, existence and uniqueness}, Commun. Contemp. Math. 11 (2009), no. 1, 131-164.
\bibitem{CPa} M. Cirant, K.R. Payne, \emph{On viscosity solutions to the Dirichlet problem for elliptic branches of inhomogeneous fully nonlinear equations},  Publ. Mat. 61 (2017), no. 2, 529-575.
\bibitem{CIL} M. G.  Crandall, H. Ishii, P.-L Lions,  \emph{User's guide to viscosity solutions of second order partial differential equations}, Bull. Amer. Math. Soc. (N.S.) 27 (1992), no. 1, 1-67. 
\bibitem{CL} A. Cutr\'i, F. Leoni, \emph{On the Liouville property for fully nonlinear equations},
Ann. Inst. H. Poincar\'e Anal. Non Lin\'eaire 17 (2000), no. 2, 219-245.
\bibitem{DFLN} D.G. de Figueiredo,  P.-L. Lions, R.D. Nussbaum,  \emph{A priori estimates and existence of positive solutions of semilinear elliptic equations}, J. Math. Pures Appl. (9) 61 (1982), no. 1, 41-63.
\bibitem{FQ} P. Felmer, A. Quaas, \emph{Positive radial solutions to a \lq\lq semilinear\rq\rq equation involving the Pucci's operator}, J. Differential Equations 199 (2004), no. 2, 376-393.
\bibitem{FQ2} P. L. Felmer, A. Quaas, \emph{On critical exponents for the Pucci's extremal operators}, Ann. Inst. H. Poincar\'e Anal. Non Lin\'eaire 20 (2003), no. 5, 843-865.
\bibitem{GLP} G. Galise, F. Leoni, F. Pacella, \emph{Existence results for fully nonlinear equations in radial
domains}, Comm. Partial Differential Equations 42 (2017), no. 5, 757-779.
\bibitem{GV} G. Galise, A. Vitolo, \emph{Removable singularities for degenerate elliptic Pucci operators}, Adv. Differential Equations 22 no. 1/2, (2017) 77-100.
\bibitem{GS}  B. Gidas, J. Spruck,  \emph{A priori bounds for positive solutions of nonlinear elliptic equations}, Comm. Partial Differential Equations 6 (1981), no. 8, 883-901.
\bibitem{GT}  D. Gilbarg, N.S. Trudinger,  \emph{Elliptic partial differential equations of second order}, Second edition. Grundlehren der Mathematischen Wissenschaften [Fundamental Principles of Mathematical Sciences], 224. Springer-Verlag, Berlin, 1983. xiii+513 pp.
\bibitem{HL} F. R. Harvey and H. Blaine Lawson Jr., \emph{Dirichlet duality and the non-linear Dirichlet
problem}, Comm. Pure Appl. Math. 62 (2009) 396-443.
\bibitem{HL2} F. R. Harvey and H. Blaine Lawson Jr., \emph{$p$-convexity, $p$-plurisubharmonicity and the Levi problem}, Indiana Univ. Math. J. 62 (2013), no. 1, 149-169.
\bibitem{IL} H. Ishii, P.-L Lions,  \emph{Viscosity  solutions  of  fully  nonlinear  second-order  elliptic  partial  differential  equations}, J. Differential Equations 83 (1) (1990) 26-78.
\bibitem{L} P.-L. Lions,  \emph{On the existence of positive solutions of semilinear elliptic equations}, SIAM Rev. 24 (1982), no. 4, 441-467.
\bibitem{OS} A.M. Oberman, L. Silvestre, \emph{The Dirichlet problem for the convex envelope}, Trans. Am. Math. Soc. 363 (2011) 5871-5886.
\bibitem{Q} A. Quaas, \emph{Existence of a positive solution to a \lq\lq semilinear\rq\rq equation involving Pucci's operator in a convex domain}, Differential Integral Equations 17 (2004), no. 5-6, 481-494.
\bibitem{QS} A. Quaas, B. Sirakov, \emph{Existence results for nonproper elliptic equations involving the Pucci operator} Comm. Partial Differential Equations 31 (2006), no. 7-9, 987-1003.
\bibitem{Sha1} J.P. Sha, \emph{$p$-convex Riemannian manifolds}, Invent. Math. 83 (1986), no. 3, 437-447. 
\bibitem{Sha2} J.P. Sha, \emph{Handlebodies and $p$-convexity}, J. Differ. Geom. 25 (1987) 353-361.
\bibitem{Vitolo} A. Vitolo, \emph{Removable singularities for degenerate elliptic equations without conditions on the growth of the solution},  Trans. Amer. Math. Soc. http://dx.doi.org/10.1090/tran/7095
\bibitem{Wu} H. Wu, \emph{Manifolds of partially positive curvature}, Indiana Univ. Math. J. 36 (1987) 525-548.
\end{thebibliography}
\end{document}